\documentclass{amsart}

\usepackage[T1]{fontenc}
\usepackage[utf8]{inputenc}
\usepackage{amssymb,mathtools}
\usepackage{microtype}
\usepackage[hidelinks]{hyperref}
\usepackage{tikz}

\theoremstyle{plain}
\newtheorem{theorem}{Theorem}[section]
\newtheorem{lemma}{Lemma}[section]
\newtheorem{proposition}{Proposition}[section]
\newtheorem{corollary}{Corollary}[section]
\newtheorem{conjecture}{Conjecture}[section]
\theoremstyle{remark}
\newtheorem{remark}{Remark}[section]

\title[A Proof of the Szeg\H{o} Conjecture]
{A Proof of the Szeg\H{o} Conjecture on Jacobi Extrema}

\author{K. Castillo}

\address{CMUC, Department of Mathematics, University of Coimbra,
3000-143 Coimbra, Portugal}

\email{math@keniercastillo.com}

\date{\today}



\begin{document}

\begin{abstract}
For Jacobi polynomials with parameters greater than $-1/2$, and with
the relative extrema enumerated from the endpoint $x=1$, the
normalised modulus at the $k$th extremum of degree $n+1$ is proved to
be strictly smaller than that at the $k$th extremum of degree $n$,
for $1\leq k\leq n$. This proves the Szeg\H{o} conjecture,
recorded in the 1975 fourth edition of his classic
monograph Orthogonal Polynomials, and strengthens it by removing
the ordering assumption on the parameters. The proof transforms the
Jacobi equation to angular form and compares Pr\"ufer amplitudes at
equal phase. Combined with a reduction of de Oliveira Filho and a separate
quadratic-transformation argument for the boundary case, the result
also settles a question concerning the Lov\'asz theta number of
spherical distance graphs in every dimension at least four.
\end{abstract}

\maketitle

\section{Introduction}

The problem originates in a conjecture of John Todd concerning the
relative extrema of Legendre polynomials; both Cooper and Szeg\H{o}
acknowledged Todd as its source. Their contributions appeared in
1950. Cooper proved that, for each fixed extremum, the asserted
decrease holds once the degree is sufficiently large
\cite{Cooper1950}, whereas Szeg\H{o} established it for every
degree and every corresponding extremum \cite{Szego1950}.
Sz\'asz extended the theorem to ultraspherical polynomials in the same
year \cite{Szasz1950}. The Jacobi question considered below is thus a
natural continuation of Szeg\H{o}'s theorem. 

We use $P_n^{(\alpha,\beta)}$ for the Jacobi polynomial of degree $n$
in the standard normalisation
$$
P_n^{(\alpha,\beta)}(1)=\frac{(\alpha+1)_n}{n!},
$$
where $(a)_n$ is the rising factorial. In the fourth edition of his
monograph, Szeg\H{o} observed that the ultraspherical
inequalities and their quadratic-transform counterpart suggested the
same conclusion for
$$
\alpha>\beta>-\frac12,
$$
and recorded this interior case as probable but open
\cite[p.~190]{Szego}.

The attribution of the conjecture requires some care. In an editorial
comment accompanying the 1982 reprint of Szeg\H{o}'s note, Askey drew
together the known boundary cases and proposed that the Jacobi
inequalities should persist throughout
$\alpha\geq\beta\geq-1/2$ \cite[p.~221]{SzegoCollected}. He later
stated the interior assertion explicitly, developed
its surrounding theory and consequences, and reported that he had
tried to prove it for twenty years
\cite[Conjecture~1, p.~23]{Askey1990}. Later literature has generally referred
to the assertion as Askey's conjecture; see \cite[pp.~1318-1319]{WongZhang},
\cite[p.~122, (4.9.6); Conjecture~24.8.6,
p.~659]{Ismail}, and
\cite[p.~95, (3.14.4)]{AskeyBateman}. Since the assertion appears in Szeg\H{o}'s monograph and, to the best
of our knowledge, no earlier formulation of it has been recorded, we
shall refer to it as Szeg\H{o}'s conjecture, while fully acknowledging
Askey's important role in stating it explicitly and bringing it to
wider attention.

The boundary information explains both the parameter range and the
strictness. The case $\alpha=\beta>-1/2$ is Sz\'asz's theorem, while
$\alpha>\beta=-1/2$ follows from it by a quadratic transformation
\cite[p.~22]{Askey1990}. At their common endpoint, the normalised
Jacobi polynomial is the Chebyshev polynomial $T_n$ of the first kind:
$$
\frac{P_n^{(-1/2,-1/2)}(x)}{P_n^{(-1/2,-1/2)}(1)}=T_n(x),
\quad
T_n(\cos\theta)=\cos(n\theta),
$$
so every relative maximum of the normalised modulus equals $1$ and
the strict inequalities become equalities.

Wong and Zhang obtained a substantial partial result for the interior
problem. They corrected Cooper's Jacobi asymptotic expansion and
proved that, for fixed $\alpha>\beta>-1/2$, all the required
inequalities hold simultaneously once the degree is sufficiently
large \cite{WongZhang}. In a separate paper they proved the reversed
inequalities for the degenerate algebraic specialisation
$P_n^{(0,-1)}$ \cite{WongZhangSIAM}. The latter result lies outside
the positive Jacobi range and does not conflict with the theorem
below.

For $\alpha,\beta>-1/2$, let
$$
1=x_{0,n}>x_{1,n}>\cdots>x_{n,n}=-1,
$$
be the points at which $\lvert P_n^{(\alpha,\beta)}\rvert$ has a relative
maximum on $[-1,1]$, with the endpoint maxima understood
one-sidedly. The derivative identity for Jacobi polynomials shows that
the $n-1$ interior points are the zeros of
$P_{n-1}^{(\alpha+1,\beta+1)}$; their strict maximality is also
verified below. In this notation the historical conjecture reads as
follows.

\begin{conjecture}[Szeg\H{o}, 1975]
\label{conj:szego-askey}
If $\alpha>\beta>-1/2$, then, for every $n\geq1$ and
$1\leq k\leq n$,
$$
\frac{\bigl|P_{n+1}^{(\alpha,\beta)}(x_{k,n+1})\bigr|}
     {P_{n+1}^{(\alpha,\beta)}(1)}
<
\frac{\bigl|P_n^{(\alpha,\beta)}(x_{k,n})\bigr|}
     {P_n^{(\alpha,\beta)}(1)}.
$$
\end{conjecture}

The main result proves this assertion in a larger parameter range.

\begin{theorem}\label{thm:main}
Let $\alpha,\beta>-1/2$. For each $n\geq1$, enumerate the points at
which $\lvert P_n^{(\alpha,\beta)}\rvert$ has a relative maximum on
$[-1,1]$, with endpoint maxima understood one-sidedly, as
$$
1=x_{0,n}>x_{1,n}>\cdots>x_{n,n}=-1.
$$
Then, for every $n\geq1$ and $1\leq k\leq n$,
$$
\frac{\bigl|P_{n+1}^{(\alpha,\beta)}(x_{k,n+1})\bigr|}
     {P_{n+1}^{(\alpha,\beta)}(1)}
<
\frac{\bigl|P_n^{(\alpha,\beta)}(x_{k,n})\bigr|}
     {P_n^{(\alpha,\beta)}(1)}.
$$
\end{theorem}

Thus the ordering $\alpha>\beta$ in the Szeg\H{o} conjecture
is unnecessary. It is important that the additional range
$-1/2<\alpha<\beta$ is not obtained from the conjectured range by a
symmetry argument. Although the reflection identity
$P_n^{(\alpha,\beta)}(-x)=(-1)^nP_n^{(\beta,\alpha)}(x)$ interchanges
the parameters, it also reverses the indexing of the
extrema and exchanges the two endpoint normalisations. It therefore
does not transform the displayed inequality for a fixed index $k$
into the inequality asserted by Theorem~\ref{thm:main}.

The proof passes from position to phase. After the substitution
$x=\cos\theta$, the normalised polynomial satisfies a singular
oscillatory equation whose first-order coefficient decreases strictly
from $+\infty$ to $-\infty$. Pr\"ufer variables separate the solution
into an amplitude and a phase. The phase is proved to be strictly
increasing; corresponding extrema of different degrees can therefore
be compared at the common phase $k\pi$. A second crossing argument
orders the inverse phases, and a sharp comparison of the first-order
coefficient after division by the degree-dependent frequency then
orders the amplitudes. This last normalised comparison is the point at
which a direct comparison in the original variable does not retain a
definite sign.

The conclusion does not extend unchanged to the entire Jacobi range
$\alpha,\beta>-1$. For example, when
$(\alpha,\beta)=(0,-3/4)$, the inequality with $n=k=1$ is reversed;
the exact calculation is recorded in Remark~\ref{rem:outside-range}.
This example does not determine the maximal parameter region in which
the comparison holds.

\section{Proof of the main theorem}

Throughout the proof, $\alpha,\beta>-1/2$ are fixed.

\subsubsection*{The angular equation}

For each $n\geq1$, set
$$
p_n(x):=
\frac{P_n^{(\alpha,\beta)}(x)}
     {P_n^{(\alpha,\beta)}(1)},
\quad
Y_n(\theta):=p_n(\cos\theta),
\quad
0\leq\theta\leq\pi,
$$
and introduce
$$
A:=\alpha+\frac12,
\quad
B:=\beta+\frac12,
\quad
\omega_n:=\sqrt{n(n+A+B)}>0.
$$
Unless otherwise indicated, a prime denotes differentiation with
respect to the displayed argument.

Since $\alpha,\beta>-\frac12$, we have $A,B>0$. The Jacobi
differential equation
$$
(1-x^2)y''
+
\bigl[\beta-\alpha-(\alpha+\beta+2)x\bigr]y'
+
n(n+\alpha+\beta+1)y
=
0
$$
and the change of variables $x=\cos\theta$ give
$$
Y_n''(\theta)
+
H(\theta)Y_n'(\theta)
+
\omega_n^2Y_n(\theta)
=
0,
\quad
0<\theta<\pi,
$$
where
$$
H(\theta)
:=
\frac{\alpha-\beta+(\alpha+\beta+1)\cos\theta}{\sin\theta}
=
A\cot\frac{\theta}{2}
-
B\tan\frac{\theta}{2}.
$$
These standard formulas for Jacobi polynomials, as well as the
endpoint values and zero distribution used below, may be found in
\cite[Chapters~IV and VI]{Szego}. Moreover,
$$
H'(\theta)
=
-\frac{A}{2\sin^2(\theta/2)}
-\frac{B}{2\cos^2(\theta/2)}
<0,
\quad
0<\theta<\pi,
$$
and
$$
\lim_{\theta\downarrow0}H(\theta)=+\infty,
\quad
\lim_{\theta\uparrow\pi}H(\theta)=-\infty.
$$
Thus
$$
H\colon(0,\pi)\longrightarrow\mathbb R
$$
is a strictly decreasing bijection.

The monotonicity and surjectivity of $H$ will play two different
roles. First, $H'<0$ will prevent the Pr\"ufer phase from losing
monotonicity. Later, the increase of $-H'/H^2$ on the interval where
$H>0$ will provide the scalar comparison between consecutive
degrees.

\subsubsection*{Pr\"ufer variables}

We use the standard Pr\"ufer amplitude and phase; see
\cite[Section~5.5]{TeschlODE}. For each $n\geq1$, the vector
$$
\left(Y_n(\theta),-\frac{Y_n'(\theta)}{\omega_n}\right)
$$
never vanishes. Indeed, simultaneous vanishing at an interior point
would give $Y_n\equiv0$ by uniqueness for the angular equation,
contrary to $Y_n(0)=1$; see
\cite[Section~2.2, Theorem~2.2]{TeschlODE}. At the endpoints,
$$
Y_n'(0)=Y_n'(\pi)=0,
\quad
Y_n(0)=1,
\quad
Y_n(\pi)=(-1)^n\frac{(\beta+1)_n}{(\alpha+1)_n}\ne0.
$$
Set
$$
r_n(\theta)
:=
\left(
Y_n(\theta)^2+\frac{Y_n'(\theta)^2}{\omega_n^2}
\right)^{1/2}.
$$
The preceding non-vanishing shows that $r_n(\theta)>0$ on
$[0,\pi]$. Hence
$$
\theta\longmapsto
\frac{1}{r_n(\theta)}
\left(
Y_n(\theta),-\frac{Y_n'(\theta)}{\omega_n}
\right)
$$
is a continuous path in $\mathbb S^1$ starting at $(1,0)$. It
therefore admits a unique continuous lift
$\varphi_n\colon[0,\pi]\to\mathbb R$ under the covering map
$s\mapsto(\cos s,\sin s)$, subject to $\varphi_n(0)=0$.
Consequently,
$$
Y_n=r_n\cos\varphi_n,
\quad
Y_n'=-\omega_n r_n\sin\varphi_n.
$$
Standard local polar-coordinate charts show that
$r_n,\varphi_n\in C^1(0,\pi)$. Differentiating the Pr\"ufer
relations and using the angular equation yields
$$
\varphi_n'
=
\omega_n
-
H\sin\varphi_n\cos\varphi_n
$$
and
$$
\frac{r_n'}{r_n}
=
-H\sin^2\varphi_n.
$$
The initial conditions are
$$
r_n(0)=1,
\quad
\varphi_n(0)=0.
$$
Thus $Y_n'=0$ precisely when $\sin\varphi_n=0$, and at such a point
$\lvert Y_n\rvert=r_n$. Once the phase is known to be strictly
increasing, the relative extrema are therefore parametrised by the
levels $\varphi_n=k\pi$.

\subsubsection*{Strict monotonicity of the phase}

\begin{lemma}\label{lem:phase-monotonicity}
For every $n\geq1$,
$$
\varphi_n'(\theta)>0,
\quad
0<\theta<\pi.
$$
\end{lemma}

\begin{proof}
Set
$$
G_n(\theta)
:=
1-\frac{H(\theta)}{\omega_n}
\sin\varphi_n(\theta)\cos\varphi_n(\theta),
$$
so that $\varphi_n'=\omega_nG_n$. We first determine the endpoint
limits of $G_n$. Since $Y_n$ is even and analytic at the origin,
write $Y_n(\theta)=1+a_n\theta^2+O(\theta^4)$. Together with
$H(\theta)=2A/\theta+O(\theta)$, the constant term in the angular
equation gives $2(2A+1)a_n+\omega_n^2=0$. Hence
$$
Y_n(\theta)
=
1-\frac{\omega_n^2}{2(2A+1)}\theta^2+O(\theta^4),
\quad
Y_n'(\theta)
=
-\frac{\omega_n^2}{2A+1}\theta+O(\theta^3).
$$
Since
$$
\sin\varphi_n\cos\varphi_n
=
-\frac{\omega_nY_nY_n'}
       {\omega_n^2Y_n^2+(Y_n')^2},
$$
and $H(\theta)=2A/\theta+O(\theta)$, it follows that
$$
\lim_{\theta\downarrow0}G_n(\theta)=\frac{1}{2A+1}>0.
$$
Similarly, with $\varepsilon=\pi-\theta$ and
$$
C_n:=Y_n(\pi)
=(-1)^n\frac{(\beta+1)_n}{(\alpha+1)_n}\ne0,
$$
the angular equation and
$H(\pi-\varepsilon)=-2B/\varepsilon+O(\varepsilon)$ may be treated
in the same way: writing
$Y_n(\pi-\varepsilon)=C_n(1+\widetilde a_n\varepsilon^2
+O(\varepsilon^4))$, the constant term gives
$2(2B+1)\widetilde a_n+\omega_n^2=0$. Thus
$$
Y_n(\pi-\varepsilon)
=
C_n\left(
1-\frac{\omega_n^2}{2(2B+1)}\varepsilon^2+O(\varepsilon^4)
\right),
$$
$$
Y_n'(\pi-\varepsilon)
=
C_n\left(
\frac{\omega_n^2}{2B+1}\varepsilon+O(\varepsilon^3)
\right).
$$
Consequently,
$$
\lim_{\theta\uparrow\pi}G_n(\theta)=\frac{1}{2B+1}>0.
$$
Differentiating its defining relation gives
$$
G_n'
=
-\frac{H'}{\omega_n}\sin\varphi_n\cos\varphi_n
-\frac{H}{\omega_n}\cos(2\varphi_n)\varphi_n'.
$$
At any zero $\theta_*$ of $G_n$, one has
$$
G_n'(\theta_*)
=
-\frac{H'(\theta_*)}{\omega_n}
\sin\varphi_n(\theta_*)\cos\varphi_n(\theta_*).
$$
The identity $G_n(\theta_*)=0$ also gives
$$
H(\theta_*)
\sin\varphi_n(\theta_*)\cos\varphi_n(\theta_*)
=
\omega_n.
$$
Let $\theta_0$ be the unique zero of $H$. Since
$G_n(\theta_0)=1$ and the two endpoint limits are positive, any zeros
of $G_n$ in $(0,\theta_0)$ or $(\theta_0,\pi)$ are confined to
compact subintervals. If a zero occurs in $(0,\theta_0)$, let
$\theta_-$ be the first one. Since $G_n>0$ immediately to its left,
$G_n'(\theta_-)\leq0$. But
$H(\theta_-)>0$, so
$\sin\varphi_n(\theta_-)\cos\varphi_n(\theta_-)>0$; since $H'<0$,
the displayed formula gives $G_n'(\theta_-)>0$, a contradiction. If
a zero occurs in $(\theta_0,\pi)$, let $\theta_+$ be the last one.
Since $G_n>0$ immediately to its right,
$G_n'(\theta_+)\geq0$, whereas $H(\theta_+)<0$ implies
$\sin\varphi_n(\theta_+)\cos\varphi_n(\theta_+)<0$ and hence
$G_n'(\theta_+)<0$. Thus
$$
G_n(\theta)>0,
\quad
0<\theta<\pi.
$$
Therefore $\varphi_n'>0$ throughout $(0,\pi)$.
\end{proof}

\subsubsection*{The phase parametrisation of the extrema}

We have shown that the phase records the oscillations of $Y_n$
without turning back. We first identify its range and the phase values
corresponding to the relative extrema.

\begin{lemma}\label{lem:phase-range}
For every $n\geq1$,
$$
\varphi_n(\pi)=n\pi.
$$
Consequently, $\varphi_n$ maps $[0,\pi]$ homeomorphically onto
$[0,n\pi]$. If
$$
\Theta_n:=\varphi_n^{-1},
$$
then
$$
x_{k,n}=\cos\Theta_n(k\pi)
$$
and
$$
\bigl\lvert p_n(x_{k,n})\bigr\rvert
=
r_n\bigl(\Theta_n(k\pi)\bigr),
\quad
0\leq k\leq n.
$$
\end{lemma}

\begin{proof}
By Lemma~\ref{lem:phase-monotonicity} and continuity,
$\varphi_n$ is strictly increasing on $[0,\pi]$. Since
$Y_n'(\pi)=0$ and $r_n(\pi)>0$, one has
$\varphi_n(\pi)=m_n\pi$ for some positive integer $m_n$.
Furthermore, $Y_n$ has exactly $n$ zeros in $(0,\pi)$. Since
$$
Y_n=r_n\cos\varphi_n,
\quad
r_n>0,
$$
and $\varphi_n$ is strictly increasing, these zeros are precisely the
inverse images of the zeros of $\cos s$ in $(0,m_n\pi)$. Hence
$m_n=n$.

The asserted homeomorphism now follows from continuity and strict
monotonicity. The second Pr\"ufer relation shows that
$Y_n'(\theta)=0$ if and only if $\sin\varphi_n(\theta)=0$.

The successive critical points, including the endpoints, are
therefore $\Theta_n(k\pi)$, $0\leq k\leq n$. At an interior critical
point $\theta_*$, the angular equation gives
$Y_n''(\theta_*)=-\omega_n^2Y_n(\theta_*)$; hence
$$
\left.\frac{d^2}{d\theta^2}Y_n(\theta)^2\right|_{\theta=\theta_*}
=-2\omega_n^2Y_n(\theta_*)^2<0.
$$
Thus every interior critical point is a strict relative maximum of
$\lvert Y_n\rvert$. At the endpoints, the expansions used in the proof of
Lemma~\ref{lem:phase-monotonicity} yield
$$
Y_n(\theta)^2
=
1-\frac{\omega_n^2}{2A+1}\theta^2+O(\theta^4)
$$
as $\theta\downarrow0$, and
$$
Y_n(\pi-\varepsilon)^2
=
C_n^2\left(
1-\frac{\omega_n^2}{2B+1}\varepsilon^2
+O(\varepsilon^4)
\right)
$$
as $\varepsilon\downarrow0$. These are strict one-sided maxima.
Hence the indexing agrees with that in
Conjecture~\ref{conj:szego-askey}, and
$$
x_{k,n}=\cos\Theta_n(k\pi).
$$
At these points,
$$
p_n(x_{k,n})
=
Y_n\bigl(\Theta_n(k\pi)\bigr)
=
(-1)^k r_n\bigl(\Theta_n(k\pi)\bigr),
$$
which proves the final assertion.
\end{proof}

The next step compares the positions at which two consecutive
solutions attain the same phase.

\subsubsection*{Comparison at equal phase}

\begin{lemma}\label{lem:phase-comparison}
For every $n\geq1$,
$$
\varphi_{n+1}(\theta)>\varphi_n(\theta),
\quad
0<\theta<\pi.
$$
Consequently,
$$
\Theta_{n+1}(s)<\Theta_n(s),
\quad
0<s\leq n\pi.
$$
\end{lemma}

\begin{proof}
The left-endpoint expansion used above gives
$$
\varphi_j(\theta)
=
\frac{\omega_j}{2A+1}\theta+O(\theta^3),
\quad \theta\downarrow0.
$$
Since $\omega_{n+1}>\omega_n$, the difference
$\varphi_{n+1}-\varphi_n$ is positive near the origin. If it vanished
in $(0,\pi)$, let $\theta_*$ be its first zero. The difference is
positive immediately to the left of $\theta_*$, so
$$
\bigl(\varphi_{n+1}-\varphi_n\bigr)'(\theta_*)\leq0.
$$
At $\theta_*$ the two phases have the same value, and their
differential equations therefore give
$$
\bigl(\varphi_{n+1}-\varphi_n\bigr)'(\theta_*)
=
\omega_{n+1}-\omega_n>0,
$$
a contradiction. Hence $\varphi_{n+1}>\varphi_n$ on $(0,\pi)$.

For $0<s<n\pi$, evaluation at $\theta=\Theta_n(s)$ and strict
monotonicity of $\varphi_{n+1}$ give
$\Theta_{n+1}(s)<\Theta_n(s)$. The same conclusion at $s=n\pi$
follows from
$\varphi_{n+1}(\pi)=(n+1)\pi>n\pi=\varphi_n(\pi)$.
\end{proof}

For $0\leq s\leq j\pi$, set
$$
R_j(s):=r_j\bigl(\Theta_j(s)\bigr).
$$
Since $\varphi_j'>0$ on $(0,\pi)$, the inverse-function theorem gives
$$
\Theta_j\in C^1(0,j\pi),
\quad
\Theta_j'(s)
=
\frac{1}{\varphi_j'(\Theta_j(s))}.
$$
Thus $R_j(s)$ is the Pr\"ufer amplitude of the degree-$j$ solution
when its phase is $s$. Write
$$
\mathbf e(s):=(\cos s,\sin s).
$$
Figure~\ref{fig:equal-phase} records the two geometric comparisons on
which the remainder of the proof rests.

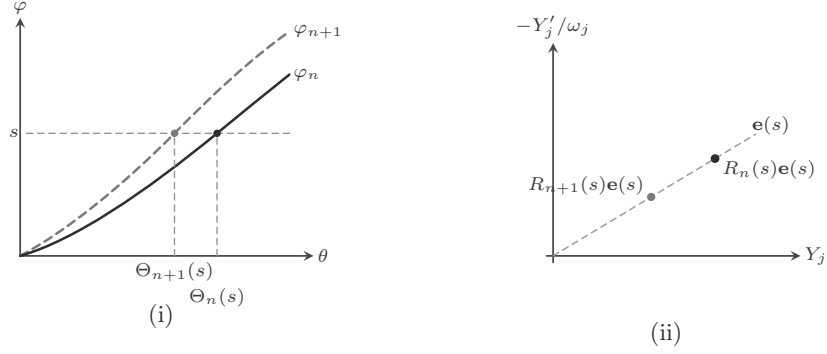
\begin{figure}[htbp]
\centering
\begin{tikzpicture}[line cap=round,line join=round]
\tikzset{
  proofaxis/.style={
    draw=black!72,
    line width=0.65pt,
    -stealth
  },
  phasen/.style={
    draw=black!82,
    line width=0.95pt
  },
  phasenext/.style={
    draw=black!52,
    line width=0.95pt,
    dash pattern=on 3.2pt off 2.1pt
  },
  guide/.style={
    draw=black!42,
    line width=0.52pt,
    dash pattern=on 2.4pt off 1.8pt
  },
  prooflabel/.style={
    font=\scriptsize,
    text=black!78,
    inner sep=1.3pt
  },
  panellabel/.style={
    font=\small,
    text=black!82
  }
}

\begin{scope}[x=0.84cm,y=0.67cm]
  \draw[proofaxis] (0,0) -- (4.62,0)
    node[right,prooflabel] {$\theta$};
  \draw[proofaxis] (0,0) -- (0,4.70)
    node[above,prooflabel] {$\varphi$};

  \draw[guide] (0.10,2.42) -- (4.27,2.42);
  \node[prooflabel,anchor=east] at (0.04,2.42) {$s$};

  \draw[phasenext]
    (0,0)
    .. controls (0.74,0.49) and (1.66,1.44) .. (2.43,2.42)
    .. controls (3.02,3.17) and (3.62,3.88) .. (4.24,4.43);

  \draw[phasen]
    (0,0)
    .. controls (0.86,0.31) and (1.96,1.25) .. (3.10,2.42)
    .. controls (3.48,2.81) and (3.85,3.19) .. (4.24,3.58);

  \draw[guide] (2.43,0) -- (2.43,2.42);
  \draw[guide] (3.10,0) -- (3.10,2.42);

  \fill[black!52] (2.43,2.42) circle (1.4pt);
  \fill[black!82] (3.10,2.42) circle (1.4pt);

  \node[prooflabel,anchor=west] at (4.24,4.43)
    {$\varphi_{n+1}$};
  \node[prooflabel,anchor=west] at (4.24,3.58)
    {$\varphi_n$};

  \node[prooflabel,anchor=north] at (2.43,-0.05)
    {$\Theta_{n+1}(s)$};
  \node[prooflabel,anchor=north] at (3.10,-0.55)
    {$\Theta_n(s)$};

  \node[panellabel] at (2.22,-1.18) {(i)};
\end{scope}

\begin{scope}[xshift=7.05cm,x=0.88cm,y=0.88cm]
  \draw[proofaxis] (-0.10,0) -- (3.72,0)
    node[right,prooflabel] {$Y_j$};
  \draw[proofaxis] (0,-0.10) -- (0,3.20)
    node[above,prooflabel] {$-Y_j'/\omega_j$};

  \draw[guide] (0,0) -- (31:3.55);
  \node[prooflabel,anchor=south west] at (31:3.42)
    {$\mathbf e(s)$};

  \fill[black!52] (31:1.72) circle (1.55pt);
  \fill[black!82] (31:2.84) circle (1.65pt);

  \node[prooflabel,anchor=south east]
    at (31:1.67) {$R_{n+1}(s)\mathbf e(s)$};
  \node[prooflabel,anchor=north west]
    at (31:2.90) {$R_n(s)\mathbf e(s)$};

  \node[panellabel] at (1.70,-1.18) {(ii)};
\end{scope}
\end{tikzpicture}
\caption{Schematic equal-phase comparison. The phase ordering in (i) gives
$\Theta_{n+1}(s)<\Theta_n(s)$. In (ii), the corresponding phase-plane
points lie on the ray generated by $\mathbf e(s)$;
Proposition~\ref{prop:amplitude-comparison}
orders their radii.}
\label{fig:equal-phase}
\end{figure}

\subsubsection*{The normalised coefficient}

The division by the degree-dependent frequency requires some care.
Although Lemma~\ref{lem:phase-comparison} and the decrease of $H$
give
$$
H\bigl(\Theta_{n+1}(s)\bigr)
>
H\bigl(\Theta_n(s)\bigr),
$$
this inequality alone is insufficient when both sides are positive,
because $\omega_{n+1}>\omega_n$. The following lemma supplies the
normalised comparison needed below.

\begin{lemma}\label{lem:q-comparison}
For $j\geq1$ and $0<s<j\pi$, define
$$
q_j(s):=
\frac{H(\Theta_j(s))}{\omega_j}.
$$
Then, for every $n\geq1$,
$$
q_{n+1}(s)>q_n(s),
\quad
0<s<n\pi.
$$
\end{lemma}

\begin{proof}
Let $\theta_0$ be the unique zero of $H$, and set
$$
s_{0,n}:=\varphi_n(\theta_0).
$$
Lemma~\ref{lem:phase-comparison} gives
$$
s_{0,n+1}>s_{0,n}.
$$
We first prove the comparison for $0<s<s_{0,n}$, where both
$q_n(s)$ and $q_{n+1}(s)$ are positive.

For $0<\theta<\theta_0$, define
$$
L(\theta):=-\frac{H'(\theta)}{H(\theta)^2}.
$$
Writing $u=\tan(\theta/2)$ and $v=u^2$, one obtains
$$
H(\theta)=\frac{A-Bu^2}{u}
$$
and
$$
2L(\theta)
=
\frac{(1+v)(A+Bv)}{(A-Bv)^2}.
$$
Since $0<v<A/B$ in this interval,
$$
\frac{d}{dv}
\left(\frac{(1+v)(A+Bv)}{(A-Bv)^2}\right)
=
\frac{A(A+3B)+B(3A+B)v}{(A-Bv)^3}
>0.
$$
Since $v=\tan^2(\theta/2)$ is strictly increasing with $\theta$, it
follows that $L$ is strictly increasing on $(0,\theta_0)$.

We next determine the order of $q_n$ near $s=0$. Put
$$
\delta:=\frac{A+3B}{6}.
$$
Since $A,B>0$, one has $\delta>0$.

Then
$$
H(\theta)
=
\frac{2A}{\theta}-\delta\theta+O(\theta^3).
$$
For either of the two frequencies under consideration, write
$\omega=\omega_j$ and $\varphi=\varphi_j$. Since $Y_j(0)=1$, one has
$Y_j>0$ near the origin, where the chosen branch of the phase is
$$
\varphi(\theta)
=
\arctan\left(-\frac{Y_j'(\theta)}{\omega Y_j(\theta)}\right).
$$
The numerator is odd and analytic, whereas the denominator is even,
analytic, and non-zero there. Thus $\varphi$ is odd and analytic, so
$$
\varphi(\theta)=\kappa\theta+\eta\theta^3+O(\theta^5).
$$
Using
$$
\sin\varphi\cos\varphi
=
\varphi-\frac23\varphi^3+O(\varphi^5)
$$
in the phase equation gives
$$
\kappa=\frac{\omega}{2A+1},
\quad
\eta
=
\frac{\frac{4A}{3}\kappa^3+\delta\kappa}{2A+3}.
$$
Inverting this expansion gives
$$
\Theta_j(s)
=
\frac{s}{\kappa}-\frac{\eta}{\kappa^4}s^3+O(s^5).
$$
Substitution into $H/\omega_j$ yields
$$
q_j(s)
=
\frac{2A}{(2A+1)s}
+
\left[
\frac{8A^2}{3(2A+1)(2A+3)}
-
\frac{3\delta(2A+1)}
     {(2A+3)\omega_j^2}
\right]s
+
O(s^3).
$$
Consequently,
$$
q_{n+1}(s)-q_n(s)
=
\frac{3\delta(2A+1)}{2A+3}
\left(
\frac1{\omega_n^2}-\frac1{\omega_{n+1}^2}
\right)s
+
O(s^3).
$$
Because $\delta>0$ and $\omega_{n+1}>\omega_n$, the right-hand side
is positive for all sufficiently small positive $s$.

Differentiating $q_j$ with respect to the phase variable gives
$$
q_j'(s)
=
\frac{H'(\Theta_j(s))}
{\omega_j^2
\bigl(1-q_j(s)\sin s\cos s\bigr)}
$$
and hence, while $q_j>0$,
$$
q_j'(s)
=
-
\frac{
L(\Theta_j(s))q_j(s)^2
}{
1-q_j(s)\sin s\cos s
}.
$$
Indeed,
$$
1-q_j(s)\sin s\cos s
=
\frac{\varphi_j'(\Theta_j(s))}{\omega_j}>0
$$
by Lemma~\ref{lem:phase-monotonicity}. Set
$$
f(s):=q_{n+1}(s)-q_n(s).
$$
Suppose, to the contrary, that $f$ has a zero in
$(0,s_{0,n})$. The local expansion shows that $f>0$ in a right
neighbourhood of $0$. Moreover,
$$
f(s_{0,n})
=
q_{n+1}(s_{0,n})-q_n(s_{0,n})
=
q_{n+1}(s_{0,n})
>0,
$$
because $q_n(s_{0,n})=0$ and $s_{0,n}<s_{0,n+1}$.
By continuity, there exist $\varepsilon_0,\varepsilon_1>0$,
chosen so that
$$
\varepsilon_0<s_{0,n}-\varepsilon_1,
$$
for which
$$
f(s)>0
\quad\text{on}\quad
(0,\varepsilon_0]
\cup
[s_{0,n}-\varepsilon_1,s_{0,n}].
$$
Consequently, every zero of $f$ in $(0,s_{0,n})$ belongs to the
compact interval
$$
[\varepsilon_0,s_{0,n}-\varepsilon_1].
$$
The zero set is therefore nonempty and compact. Let $s_*$ be its
least element. Then
$$
f(s)>0\quad(0<s<s_*),
\quad
f(s_*)=0,
$$
and hence
$$
f'(s_*)\leq0.
$$
Writing
$$
q_{n+1}(s_*)=q_n(s_*)=:q,
$$
we have $q>0$, since $s_*<s_{0,n}$. 

Lemma~\ref{lem:phase-comparison} and the strict increase of $L$ give
$$
L\bigl(\Theta_{n+1}(s_*)\bigr)
<
L\bigl(\Theta_n(s_*)\bigr).
$$
On the other hand,
$$
\bigl(q_{n+1}-q_n\bigr)'(s_*)
=
\frac{q^2}{
1-q\sin s_*\cos s_*
}
\left[
L\bigl(\Theta_n(s_*)\bigr)
-
L\bigl(\Theta_{n+1}(s_*)\bigr)
\right]
>0.
$$
This is a contradiction. Hence
$$
q_{n+1}(s)>q_n(s),
\quad
0<s\leq s_{0,n}.
$$
For $s_{0,n}<s<\min\{s_{0,n+1},n\pi\}$, one has
$$
q_n(s)<0\leq q_{n+1}(s).
$$
If $s_{0,n+1}<n\pi$, the same strict inequality holds at
$s=s_{0,n+1}$ by the signs. Finally, if
$s_{0,n+1}<s<n\pi$, then both values of $H$ are negative.
Since
$$
\Theta_{n+1}(s)<\Theta_n(s)
$$
and $H$ is strictly decreasing,
$$
\frac{H(\Theta_{n+1}(s))}{\omega_{n+1}}
>
\frac{H(\Theta_{n+1}(s))}{\omega_n}
>
\frac{H(\Theta_n(s))}{\omega_n}.
$$
The three cases prove the result.
\end{proof}

\subsubsection*{Amplitude comparison and conclusion}

\begin{proposition}[Equal-phase amplitude comparison]
\label{prop:amplitude-comparison}
For every $n\geq1$,
$$
R_{n+1}(s)<R_n(s),
\quad
0<s\leq n\pi.
$$
\end{proposition}

\begin{proof}
For $0<s<n\pi$, the chain rule and the two Pr\"ufer equations give
$$
\frac{d}{ds}\log R_n(s)
=
-
\frac{
q_n(s)\sin^2s
}{
1-q_n(s)\sin s\cos s
}.
$$
Writing $\nu(s)=\sin s\cos s$ and subtracting the two logarithmic
derivatives, we obtain
$$
\frac{d}{ds}
\log\frac{R_{n+1}(s)}{R_n(s)}
=
-
\frac{
\bigl(q_{n+1}(s)-q_n(s)\bigr)\sin^2s
}{
\bigl(1-q_{n+1}(s)\nu(s)\bigr)
\bigl(1-q_n(s)\nu(s)\bigr)
}.
$$
Both factors in the denominator are positive. Hence
Lemma~\ref{lem:q-comparison} shows that this derivative is negative
unless $s$ is an integral multiple of $\pi$. The function
$$
F(s):=\log\frac{R_{n+1}(s)}{R_n(s)}
$$
is continuous at the origin and $F(0)=0$, since
$R_n(0)=R_{n+1}(0)=1$. On each compact subinterval of $(0,n\pi)$,
the displayed derivative is continuous, non-positive, and strictly
negative away from the discrete set $\pi\mathbb Z$. Integration over
any non-degenerate compact subinterval therefore shows that $F$ is
strictly decreasing there, and continuity at the origin gives
$$
R_{n+1}(s)<R_n(s),
\quad
0<s<n\pi.
$$
Both amplitudes are continuous at the final phase. Moreover, the
logarithmic ratio is nonincreasing on $(0,n\pi)$. Thus, for any fixed
$s_0\in(0,n\pi)$, continuity and passage to the limit as
$s\uparrow n\pi$ give
$$
\log\frac{R_{n+1}(n\pi)}{R_n(n\pi)}
\leq
\log\frac{R_{n+1}(s_0)}{R_n(s_0)}
<0.
$$
\end{proof}

\begin{proof}[Proof of Theorem~\ref{thm:main}]
By Lemma~\ref{lem:phase-range}, for $1\leq k\leq n$,
$$
\frac{
\bigl|P_{n+1}^{(\alpha,\beta)}(x_{k,n+1})\bigr|
}{
P_{n+1}^{(\alpha,\beta)}(1)
}
=
R_{n+1}(k\pi)
<
R_n(k\pi)
=
\frac{
\bigl|P_n^{(\alpha,\beta)}(x_{k,n})\bigr|
}{
P_n^{(\alpha,\beta)}(1)
}.
$$
The strict inequality is Proposition~\ref{prop:amplitude-comparison},
and the theorem follows.
\end{proof}

\begin{remark}\label{rem:outside-range}
The conclusion cannot be extended to the entire Jacobi range
$\alpha,\beta>-1$. For
$\alpha,\beta>-1$ and $n=1$, the first non-trivial extremum counted
from $1$ is the left endpoint. For $n=2$, the first interior critical
point is
$$
x_*:=\frac{\beta-\alpha}{\alpha+\beta+4},
$$
and direct calculation gives
$$
\frac{\bigl\lvert P_1^{(\alpha,\beta)}(-1)\bigr\rvert}
     {P_1^{(\alpha,\beta)}(1)}
=\frac{\beta+1}{\alpha+1},
\quad
\frac{\bigl\lvert P_2^{(\alpha,\beta)}(x_*)\bigr\rvert}
     {P_2^{(\alpha,\beta)}(1)}
=\frac{\beta+2}{(\alpha+1)(\alpha+\beta+4)}.
$$
At $(\alpha,\beta)=(0,-3/4)$ these values are $1/4$ and $5/13$,
respectively, so the required inequality is reversed. At the
Chebyshev corner $\alpha=\beta=-1/2$, by contrast, every
corresponding value equals $1$.
\end{remark}

\section{A spherical consequence}
\label{sec:spherical}

The distinction between the open interior and its classical boundary
arises in an application due to de Oliveira Filho. If
$G(S^{d-1},t)$ denotes the graph on the unit sphere in $\mathbb R^d$
in which two points are adjacent when their inner product is $t$, he
asked whether, with $\vartheta$ denoting his extension of the
Lov\'asz theta number to these infinite distance graphs
\cite[Section~3.3]{OliveiraFilho},
$$
\lim_{t\uparrow1}\vartheta\bigl(G(S^{d-1},t)\bigr)
\leq
\vartheta\bigl(G(S^{d-1},u)\bigr),
\quad -1\leq u<1,
$$
and showed that the Jacobi-extremum comparison stated below is
sufficient for an affirmative answer \cite[Section~3.5d, Question~3, pp.~47--48]{OliveiraFilho}. The case
$d=4$ requires the following classical boundary result.

\begin{proposition}\label{prop:quadratic-boundary}
Let $\gamma>-1/2$, and for each $n\geq1$ let
$$
1=\xi_{0,n}>\xi_{1,n}>\cdots>\xi_{n,n}=-1
$$
be the points at which $\lvert P_n^{(\gamma,-1/2)}\rvert$ has a
relative maximum, with the endpoints understood one-sidedly. Then,
for $1\leq k\leq n$,
$$
\frac{\bigl\lvert P_{n+1}^{(\gamma,-1/2)}(\xi_{k,n+1})\bigr\rvert}
     {P_{n+1}^{(\gamma,-1/2)}(1)}
<
\frac{\bigl\lvert P_n^{(\gamma,-1/2)}(\xi_{k,n})\bigr\rvert}
     {P_n^{(\gamma,-1/2)}(1)}.
$$
\end{proposition}

\begin{proof}
The normalised quadratic transformation \cite[(4.1.5)]{Szego} is
$$
\frac{P_n^{(\gamma,-1/2)}(2u^2-1)}
     {P_n^{(\gamma,-1/2)}(1)}
=
\frac{P_{2n}^{(\gamma,\gamma)}(u)}
     {P_{2n}^{(\gamma,\gamma)}(1)},
\quad 0\leq u\leq1.
$$
Because $u\mapsto2u^2-1$ maps $[0,1]$ increasingly onto $[-1,1]$,
the $k$th Jacobi extremum corresponds to the $k$th ultraspherical
extremum in $[0,1]$ when both are enumerated from the endpoint $1$.
This remains true for $k=n$. Put $U=P_{2n}^{(\gamma,\gamma)}$.
By parity, $U'(0)=0$. If $U(0)=0$, then $0$ would be a multiple
zero of $U$, contradicting the simplicity of the zeros of Jacobi
polynomials. Hence $U(0)\ne0$. The differential equation for $U$ gives
$$
U''(0)=-2n(2n+2\gamma+1)U(0),
$$
so $(U^2)''(0)<0$. Thus $u=0$ is a strict relative maximum of
$\lvert U\rvert$. If $M_{k,m}$
denotes the normalised modulus at the $k$th
ultraspherical extremum of degree $m$, the two quantities to be
compared are therefore $M_{k,2n+2}$ and $M_{k,2n}$. Two applications
of Sz\'asz's theorem \cite{Szasz1950} give
$$
M_{k,2n+2}<M_{k,2n+1}<M_{k,2n},
\quad 1\leq k\leq n,
$$
as required.
\end{proof}

\begin{corollary}\label{cor:spherical-theta}
Let $d\geq4$, and let $G(S^{d-1},t)$ and $\vartheta$ be as above.
Then, for every $-1\leq u<1$,
$$
\lim_{t\uparrow1}\vartheta\bigl(G(S^{d-1},t)\bigr)
\leq
\vartheta\bigl(G(S^{d-1},u)\bigr).
$$
\end{corollary}

\begin{proof}
Put
$$
(\alpha,\beta)
=
\left(\frac{d-3}{2},\frac{d-5}{2}\right),
$$
and let $z_{1,m}$ denote the rightmost interior point at which
$\lvert P_m^{(\alpha,\beta)}\rvert$ has a relative maximum. De
Oliveira Filho showed that it suffices to prove
$$
\frac{\bigl\lvert P_{m+1}^{(\alpha,\beta)}(z_{1,m+1})\bigr\rvert}
     {P_{m+1}^{(\alpha,\beta)}(1)}
<
\frac{\bigl\lvert P_m^{(\alpha,\beta)}(z_{1,m})\bigr\rvert}
     {P_m^{(\alpha,\beta)}(1)},
\quad m\geq2;
$$
see \cite[(3.40)--(3.43), p.~48]{OliveiraFilho}. If $d\geq5$, both
parameters exceed $-1/2$, and this is Theorem~\ref{thm:main} with
$k=1$. For $d=4$, the parameters are $(1/2,-1/2)$, and the same
comparison follows from Proposition~\ref{prop:quadratic-boundary}.
\end{proof}

For $d=3$, the parameters are $(0,-1)$. De Oliveira Filho's
calculation gives
\[
\lim_{t\uparrow1}\vartheta\bigl(G(S^2,t)\bigr)
>
\vartheta\bigl(G(S^2,-1/3)\bigr);
\]
see \cite[Theorem~3.5 and p.~48]{OliveiraFilho}. Thus Question~3
has a negative answer for $d=3$, whereas
Corollary~\ref{cor:spherical-theta} gives an affirmative answer for
every $d\geq4$.

\section{Concluding remarks}

The proof rests on two scalar monotonicity properties. The inequality
$H'<0$ forces each Pr\"ufer phase to increase strictly. Comparison of
the phase equations then orders consecutive phases and their
inverses, while the endpoint expansion, together with the increase of
$-H'/H^2$ before the zero of $H$, yields the frequency-normalised
comparison required for the amplitudes. Neither argument uses the
sign of $\alpha-\beta$; only
$A,B>0$ are required.
This explains the extension from $\alpha>\beta>-1/2$ to the full
region $\alpha,\beta>-1/2$.

Strictness cannot persist throughout the closed region: at
$\alpha=\beta=-1/2$ the normalised polynomials are Chebyshev
polynomials and all relative maxima have modulus one. Boundary
families therefore require separate treatment, as illustrated by
Proposition~\ref{prop:quadratic-boundary}.

The spherical consequence shows that the same extremal comparison
also governs a concrete question for the Lov\'asz theta number of
spherical distance graphs.

\section*{Acknowledgements}
The author gratefully acknowledges that the idea of applying Pr\"ufer variables to the present problem arose while reading Kiselev's beautiful paper \cite{Kiselev2005} in connection with an entirely different question. The author acknowledges financial support from the Centre for
Mathematics of the University of Coimbra (CMUC), funded by the
Portuguese Foundation for Science and Technology (FCT), under the
projects
\href{https://doi.org/10.54499/UID/00324/2025}{UID/00324/2025} and
UID/PRR/00324/2025. The author also acknowledges financial support
from the FCT under grant
\href{https://doi.org/10.54499/2022.00143.CEECIND/CP1714/CT0002}
{2022.00143.CEECIND}.

\end{document}